\documentclass[a4paper,reqno]{amsart}

\PII{}
\makeatletter
\def\@setcopyright{\@empty}
\makeatother

\newcommand{\Lp}{L_p[0,2\pi]}
\newcommand{\allp}{1<p<\infty}
\newcommand\seq[2]{\{#1_{#2}\}_{#2=1}^\infty}

\newcommand{\w}{\omega_k(f,t)_p}
\newcommand{\wpar}[1]{\omega_k\left(f,#1\right)_p}
\newcommand{\Np}{N(p,\theta,r,\lambda,\varphi)}

\theoremstyle{plain}
\newtheorem{thm}{Theorem}[subsection]
\newtheorem{lmm}{Lemma}[subsection]

\theoremstyle{remark}
\newtheorem{rem}{Remark}[subsection]
\numberwithin{rem}{thm}
\numberwithin{rem}{lmm}

\newcounter{tempenumi}

\newcounter{const}[subsection]
\numberwithin{const}{thm}
\numberwithin{const}{lmm}
\makeatletter
\newcommand{\Cn}[1][]{%
  \stepcounter{const}C_{\theconst}%
  \@ifnotempty{#1}{\newcounter{#1}\setcounter{#1}{\arabic{const}}}}
\makeatother

\newcommand{\lastC}{C_{\theconst}}
\newcommand{\prevC}[1][1]{%
	{\countdef\n=255
	 \n=\theconst
	 \advance\n by-#1
	 C_{\number\n}}}

\numberwithin{equation}{subsection}

\renewcommand{\theconst}{\arabic{const}}

\begin{document}

\title[On monotone Fourier coefficients of a function\dots]
	{On monotone Fourier coefficients of a function belonging
	 to Nikol'ski\u{\i}--Besov classes}
\author{M.~Q.\ Berisha}
\author{F.~M.\ Berisha}
\address{Faculty of Mathematics and Sciences\\
	University of Prishtina\\
	N\"e\-na Te\-re\-z\"e~5\\
	10000 Prishtina\\
	Kosovo}
\email{faton.berisha@uni-pr.edu}

\date{}

\begin{abstract}
	In this paper,
	necessary and sufficient conditions
	on terms of monotone Fourier coefficients
	for a function to belong to a Nikol'ski\u{\i}--Besov type class
	are given.
\end{abstract}

\maketitle

\subsection{}

Let $f\in\Lp$, $\allp$, be a $2\pi$-periodic function
having a cosine Fourier series with monotone coefficients,
i.e.
\begin{displaymath}
	f(x)\sim\sum_{n=0}^\infty a_n\cos nx, \quad a_n\downarrow0.
\end{displaymath}
and~$\w$ the modulus of smoothness of order~$k$ in~$\Lp$ metrics
of the function $f$, i.e.
\begin{displaymath}
	\w=\sup_{|h|\le t}\|\Delta_h^k f\|_p,
\end{displaymath}
where is
\begin{displaymath}
	\Delta_h^k f(x)=\sum_{\nu=0}^k(-1)^{k-\nu}\binom k\nu f(x+\nu h).
\end{displaymath}

We say that a $2\pi$--periodic function~$f$
belongs to the Nikol'ski\u{\i}--Besov class $\Np$, $\allp$,
if the following conditions are satisfied
\begin{enumerate}
	\item $f\in\Lp$;
	\item\label{it:parameters}Numbers~$\theta$, $r$,~$\lambda$
	  belong to the interval $(0,\infty)$, and~$k$ is an integer
	  satisfying $k>r+\lambda$;
	\item The following inequality holds true
		\begin{displaymath}
			\biggl(
				\int_0^\delta t^{-r\theta-1}\w^\theta\,dt
				+\delta^{\lambda\theta}\int_\delta^1
					t^{-(r+\lambda)\theta-1}\w^\theta\,dt
			\biggr)^{1/\theta}
			\le C\varphi(\delta),
		\end{displaymath}
	\setcounter{tempenumi}{\theenumi}
\end{enumerate}
while the function~$\varphi$ satisfies the conditions
\begin{enumerate}
	\setcounter{enumi}{\thetempenumi}
	\item\label{it:phi-continuous}$\varphi$~is a non-negative
	  continuous function on $(0,1)$ and $\varphi\ne0$;
	\item For every~$\delta_1$, $\delta_2$ such that
		$0\le\delta_1\le\delta_2\le1$ holds
		$\varphi(\delta_1)\le\Cn\varphi(\delta_2)$;
	\item\label{it:phi-2}For every~$\delta$ such that
	  $0\le\delta\le\frac12$
		holds $\varphi(2\delta)\le\Cn\varphi(\delta)$,
\end{enumerate}
where constants%
\footnote{Without mentioning it explicitly,
	we will consider all the constants positive.%
}%
~$C$, $\prevC$ and~$\lastC$
do not depend on~$\delta_1$, $\delta_2$ and~$\delta$.

A more detailed approach to the classes $\Np$
is given in~\cite{lakovic:mat-87}
(see also~\cite[p.~298]{besov-i-n:integralnye}).
In our paper we give the necessary and sufficient condition
in terms of monotone Fourier coefficients
for a function $f\in\Lp$ to belong to the class $\Np$.

\subsection{}

Now we formulate our results.

\begin{thm}\label{th:Np-w}
	A function~$f$ belongs to the class $\Np$
	if and only if%
	\footnote{Here and below we assume that the parameters~$\theta$,
		$r$, $\lambda$ and~$k$
		satisfy the condition~\ref{it:parameters},
		and the function~$\varphi$ satisfies the conditions
		\ref{it:phi-continuous}--\ref{it:phi-2}
		of the definition of the class $\Np$.}
	\begin{multline}\label{eq:Np-w}
		\biggl(
			\sum_{\nu=n+1}^\infty \wpar{\frac1\nu}^\theta \nu^{r\theta-1}
			+n^{-\lambda\theta}\sum_{\nu=1}^n
				\wpar{\frac1\nu}^\theta \nu^{(r+\lambda)\theta-1}
		\biggr)^{1/\theta}\\
		\le C\varphi\left(\frac1n\right),
	\end{multline}
	where constant~$C$ does not depend on~$n$.
\end{thm}

\begin{thm}\label{th:Np-monotone}
	For a function $f\in\Lp$, $\allp$,
	such that
	\begin{equation}\label{eq:f}
		f(x)\sim\sum_{\nu=1}^\infty a_\nu\cos\nu x, \quad a_\nu\downarrow0,
	\end{equation}
	to belong to the class $\Np$
	it is necessary and sufficient
	that its Fourier coefficients satisfy the condition
	\begin{displaymath}
		\biggl(
			\sum_{\nu=n+1}^\infty a_\nu^\theta \nu^{r\theta+\theta-\theta/p-1}
			+n^{-\lambda\theta}\sum_{\nu=1}^n a_\nu^\theta
				\nu^{r\theta+\lambda\theta+\theta-\theta/p-1}
		\biggr)^{1/\theta}
		\le C\varphi\left(\frac 1n\right),
	\end{displaymath}
	where constant~$C$ does not depend on~$n$.
\end{thm}

\begin{rem}
	Put $\varphi(\delta)=\delta^\alpha$, $0<\alpha<\lambda$,
	in the definition of the class $\Np$,
	we obtain~\cite{lakovic:mat-87}
	the Nikol'ski\u{\i} class~$H_p^{r+\alpha}$.
	Thus Theorems~\ref{th:Np-w} and~\ref{th:Np-monotone}
	give the single coefficient condition
	\begin{displaymath}
		a_\nu\le\frac C{\nu^{r+\alpha+1-\frac1p}},
	\end{displaymath}
	for $f\in H_p^{r+\alpha}$,
	given in~\cite{konyushkov:izv-57}
	(see also~\cite{berisha:glas-81}),
	where the function~$f$ is given by~\eqref{eq:f}.
\end{rem}

\begin{rem}
	If $\varphi(\delta)\ge C$,
	then we obtain~\cite{lakovic:mat-87}
	the Besov class~$B_p^{\theta r}$.
	Thus Theorems~\ref{th:Np-w} and~\ref{th:Np-monotone}
	give the necessary and sufficient condition
	\begin{displaymath}
		\sum_{\nu=1}^\infty a_\nu^\theta \nu^{r\theta+\theta-\theta/p-1}
		<\infty
	\end{displaymath}
	for $f\in B_p^{\theta r}$,
	given in~\cite{potapov-b:publ-79}
	(see also~\cite{berisha:serdica-85}),
	where the function~$f$ is given by~\eqref{eq:f}.
\end{rem}

\subsection{}

In order to establish our results,
we use the following lemmas.

\begin{lmm}\label{lm:jensen}
	Let $0<\alpha<\beta<\infty$ and $a_\nu\ge0$. The following
	inequality holds true
	\begin{displaymath}
		\biggl(\sum_{\nu=1}^n a_\nu^\beta\biggr)^{1/\beta}
		\le\biggl(\sum_{\nu=1}^n a_\nu^\alpha\biggr)^{1/\alpha}.
	\end{displaymath}
\end{lmm}

Proof of the lemma
is due to Jensen~\cite[p.~43]{hardy-l-p:inequalities}.

\begin{lmm}\label{lm:lp}
	Let $\seq a\nu$ be a sequence of non-negative numbers, $\alpha>0$,
	$\lambda$ a real number,
	$m$ and~$n$ positive integers such that $m<n$.
	Then
	\begin{enumerate}
	\item for $1\le p<\infty$ the following equalities hold
		\begin{displaymath}
			\sum_{\mu=m}^n \mu^{\alpha-1}
				\biggl(\sum_{\nu=\mu}^n a_\nu \nu^\lambda\biggr)^p
			\le\Cn\sum_{\mu=m}^n \mu^{\alpha-1}(a_\mu \mu^{\lambda+1})^p,
		\end{displaymath}
		\begin{displaymath}
			\sum_{\mu=m}^n \mu^{-\alpha-1}
				\biggl(\sum_{\nu=m}^\mu a_\nu \nu^\lambda\biggr)^p
			\le\Cn\sum_{\mu=m}^n \mu^{-\alpha-1}(a_\mu \mu^{\lambda+1})^p;
		\end{displaymath}
	\item for $0<p\le1$ the following equalities hold
		\begin{displaymath}
			\sum_{\mu=m}^n \mu^{\alpha-1}
				\biggl(\sum_{\nu=\mu}^n a_\nu \nu^\lambda\biggr)^p
			\ge\Cn\sum_{\mu=m}^n \mu^{\alpha-1}(a_\mu \mu^{\lambda+1})^p,
		\end{displaymath}
		\begin{displaymath}
			\sum_{\mu=m}^n \mu^{-\alpha-1}
				\biggl(\sum_{\nu=m}^\mu a_\nu \nu^\lambda\biggr)^p
			\ge\Cn\sum_{\mu=m}^n \mu^{-\alpha-1}(a_\mu \mu^{\lambda+1})^p,
		\end{displaymath}
	\end{enumerate}
	where constants~$\prevC[3]$, $\prevC[2]$, $\prevC$ and~$\lastC$
	depend only on numbers~$\alpha$, $\lambda$ and~$p$,
	and do not depend on~$m$, $n$
	as well as on the sequence $\seq a\nu$.
\end{lmm}

Proof of the lemma is given in~\cite[p.~308]{hardy-l-p:inequalities}.

We write $a_\nu\downarrow$
if $\seq a\nu$ is a monotone--decreasing sequence of non-negative numbers,
i.e.\ if $a_\nu\ge a_{\nu+1}\ge0$ $(\nu=1,2,\dotsc)$.

\begin{lmm}\label{lm:lp-converse}
	Let $a_\nu\downarrow$, $\alpha>0$, $\lambda$ a real number,
	$m$ and~$n$ positive integers.
	Then
	\begin{enumerate}
	\item for $1\le p<\infty$, $n\ge16m$ the following equalities hold
		\begin{displaymath}
			\sum_{\mu=m}^n \mu^{\alpha-1}
				\biggl(\sum_{\nu=\mu}^n a_\nu \nu^\lambda\biggr)^p
			\ge\Cn\sum_{\mu=8m}^n \mu^{\alpha-1}(a_\mu \mu^{\lambda+1})^p,
		\end{displaymath}
		\begin{displaymath}
			\sum_{\mu=m}^n \mu^{-\alpha-1}
				\biggl(\sum_{\nu=m}^\mu a_\nu \nu^\lambda\biggr)^p
			\ge\Cn\sum_{\mu=4m}^n \mu^{-\alpha-1}(a_\mu \mu^{\lambda+1})^p;
		\end{displaymath}
	\item for $0<p\le1$, $n\ge4m$ the following equalities hold
		\begin{displaymath}
			\sum_{\mu=4m}^n \mu^{\alpha-1}
				\biggl(\sum_{\nu=\mu}^n a_\nu \nu^\lambda\biggr)^p
			\le\Cn\sum_{\mu=m}^n \mu^{\alpha-1}(a_\mu \mu^{\lambda+1})^p,
		\end{displaymath}
		\begin{displaymath}
			\sum_{\mu=4m}^n \mu^{-\alpha-1}
				\biggl(\sum_{\nu=4m}^\mu a_\nu \nu^\lambda\biggr)^p
			\le\Cn\sum_{\mu=m}^n \mu^{-\alpha-1}(a_\mu \mu^{\lambda+1})^p,
		\end{displaymath}
	\end{enumerate}
	where constants~$\prevC[3]$, $\prevC[2]$, $\prevC$ and~$\lastC$
	depend only on numbers~$\alpha$, $\lambda$ and~$p$,
	and do not depend on~$m$, $n$
	as well as on the sequence $\seq a\nu$.
\end{lmm}

Proof of the lemma is given in~\cite{berisha:lp-conv}.

\begin{lmm}\label{lm:lp-complete}
	Let $a_\nu\downarrow$, $\alpha>0$, $\lambda$ a real number,
	$m$ and~$n$ positive integers.
	For $0<p<\infty$ the following inequalities hold
	\begin{displaymath}
		\Cn\sum_{\mu=1}^n \mu^{\alpha-1}(a_\mu \mu^{\lambda+1})^p
		\le\sum_{\mu=1}^n \mu^{\alpha-1}
			\biggl(\sum_{\nu=\mu}^n a_\nu \nu^\lambda\biggr)^p
		\le\Cn\sum_{\mu=1}^n \mu^{\alpha-1}(a_\mu \mu^{\lambda+1})^p,
	\end{displaymath}
	\begin{displaymath}
		\Cn\sum_{\mu=1}^n \mu^{-\alpha-1}(a_\mu \mu^{\lambda+1})^p
		\le\sum_{\mu=1}^n \mu^{-\alpha-1}
			\biggl(\sum_{\nu=1}^\mu a_\nu \nu^\lambda\biggr)^p
		\le\Cn\sum_{\mu=1}^n \mu^{-\alpha-1}(a_\mu \mu^{\lambda+1})^p,
	\end{displaymath}
	where constants~$\prevC[3]$, $\prevC[2]$, $\prevC$ and~$\lastC$
	depend only on numbers~$\alpha$, $\lambda$ and~$p$,
	and do not depend on~$m$, $n$
	as well as on the sequence $\seq a\nu$.
\end{lmm}

The lemma is also proved in~\cite{berisha:lp-conv}.

\begin{lmm}\label{lm:anu-w}
	Let $f\in\Lp$ for a fixed~$p$ from the interval $1<p<\infty$
	and let
	\begin{displaymath}
		f(x)\sim\sum_{\nu=1}^\infty a_\nu\cos\nu x, \quad a_\nu\downarrow0.
	\end{displaymath}
	The following inequalities hold
	\begin{multline*}
		\Cn\frac1{n^k}
					\biggl(\sum_{\nu=1}^n a_\nu^p \nu^{(k+1)p-2}\biggr)^{1/p}
				+\biggl(\sum_{\nu=n+1}^\infty a_\nu^p \nu^{p-2}\biggr)^{1/p}
			\le\wpar{\frac1n}\\
		\le\Cn\frac1{n^k}
					\biggl(\sum_{\nu=1}^n a_\nu^p \nu^{(k+1)p-2}\biggr)^{1/p}
				+\biggl(\sum_{\nu=n+1}^\infty a_\nu^p \nu^{p-2}\biggr)^{1/p},
	\end{multline*}
	where constants~$\prevC$ and~$\lastC$
	do not depend on~$n$ and~$f$.
\end{lmm}

The lemma is proved in~\cite{potapov-b:publ-79}.
\subsection{}

Now we prove our results.

\begin{proof}[Proof of Theorem~\ref{th:Np-w}]
	\setcounter{const}{0}
	Put
	\begin{displaymath}
		I_1=\int_0^\frac1{n+1} t^{-r\theta-1}\w^\theta\,dt, \quad
		I_2=\int_\frac1{n+1}^1 t^{-(r+\lambda)\theta-1}\w^\theta\,dt.
	\end{displaymath}
	We have~\cite[p.~55]{hardy-l-p:inequalities}
	\begin{multline*}
		I_1=\int_0^\frac1{n+1} t^{-r\theta-1}\w^\theta\,dt
			=\sum_{\nu=n+1}^\infty
				\int_\frac1{\nu+1}^\frac1\nu t^{-r\theta-1}\w^\theta\,dt\\
		\le\sum_{\nu=n+1}^\infty \wpar{\frac1\nu}^\theta
				\int_\frac1{\nu+1}^\frac1\nu t^{-r\theta-1}\,dt
			\le\Cn\sum_{\nu=n+1}^\infty \wpar{\frac1\nu}^\theta
				\nu^{r\theta-1}
	\end{multline*}
	and,
	taking into account properties
	of modulus of smoothness~\cite[p.~116]{timan:priblizheniya},
	\begin{displaymath}
		I_1\ge\sum_{\nu=n+1}^\infty \wpar{\frac1{\nu+1}}^\theta
				\int_\frac1{\nu+1}^\frac1\nu t^{-r\theta-1}\,dt
			\ge\Cn\sum_{\nu=n+1}^\infty \wpar{\frac1\nu}^\theta
				\nu^{r\theta-1}.
	\end{displaymath}
	
	In an analogous way we estimate
	\begin{displaymath}
		I_2\le\sum_{\nu=1}^n \wpar{\frac1\nu}^\theta
				\int_\frac1{\nu+1}^\frac1\nu t^{-(r+\lambda)\theta-1}\,dt
			\le\Cn\sum_{\nu=1}^n \wpar{\frac1\nu}^\theta
				\nu^{(r+\lambda)\theta-1}
	\end{displaymath}
	and
	\begin{displaymath}
		I_2\ge\sum_{\nu=1}^n \wpar{\frac1{\nu+1}}^\theta
				\int_\frac1{\nu+1}^\frac1\nu t^{-(r+\lambda)\theta-1}\,dt
			\ge\Cn\sum_{\nu=1}^n \wpar{\frac1\nu}^\theta
				\nu^{(r+\lambda)\theta-1}.
	\end{displaymath}
	
	Let $f\in\Np$.
	For a positive integer~$n$
	we put $\delta=\frac1{n+1}$.
	Then we have
	\begin{multline*}
		I^\theta=I_1+\delta^{\lambda\theta}I_2\\
		\ge\Cn\biggl(
			\sum_{\nu=n+1}^\infty \wpar{\frac1\nu}^\theta \nu^{r\theta-1}
			+n^{-\lambda\theta}\sum_{\nu=1}^n \wpar{\frac1\nu}^\theta
				\nu^{(r+\lambda)\theta-1}
		\biggr).
	\end{multline*}
	Hence we obtain
	\begin{multline*}
		J=\biggl(
				\sum_{\nu=n+1}^\infty \wpar{\frac1\nu}^\theta \nu^{r\theta-1}
				+n^{-\lambda\theta}\sum_{\nu=1}^n \wpar{\frac1\nu}^\theta
					\nu^{(r+\lambda)\theta-1}
			\biggr)^{1/\theta}\\
		\le\Cn I\le\Cn\varphi(\delta)=\lastC\varphi\left(\frac1{n+1}\right)
			\le\Cn\varphi\left(\frac1n\right),
	\end{multline*}
	which proves inequality~\eqref{eq:Np-w}.
	
	Now we suppose that inequality~\eqref{eq:Np-w} holds.
	For $\delta\in(0,1)$ we choose the positive integer~$n$
	satisfying $\frac1{n+1}<\delta\le\frac1n$.
	Then,
	taking into consideration the estimates from above
	for~$I_1$ and~$I_2$
	we have
	\begin{multline*}
		I^\theta=\int_0^\frac1{n+1} t^{-r\theta-1}\w^\theta\,dt
			+\int_\frac1{n+1}^\delta t^{-r\theta-1}\w^\theta\,dt\\
		+\delta^{\lambda\theta}
				\int_\delta^1 t^{-(r+\lambda)\theta-1}\w^\theta\,dt
			\le I_1+\delta^{\lambda\theta}I_2\\
		\le\Cn\biggl(
			\sum_{\nu=n+1}^\infty \wpar{\frac1\nu}^\theta \nu^{r\theta-1}
			+n^{-\lambda\theta}\sum_{\nu=1}^n \wpar{\frac1\nu}^\theta
				\nu^{(r+\lambda)\theta-1}
		\biggr).
	\end{multline*}
	Whence
	\begin{displaymath}
		I\le\Cn J\le\Cn\varphi\left(\frac1n\right)
		\le\Cn\varphi\left(\frac1{2n}\right)\le\Cn\varphi(\delta),
	\end{displaymath}
	implying $f\in\Np$.
	
	Proof of Theorem~\ref{th:Np-w} is completed.
\end{proof}

\begin{proof}[Proof of Theorem~\ref{th:Np-monotone}]
	\setcounter{const}{0}
	Theorem~\ref{th:Np-w} implies that the condition $f\in\Np$
	is equivalent to the condition
	\begin{displaymath}
		\sum_{\nu=n+1}^\infty \wpar{\frac1\nu}^\theta \nu^{r\theta-1}
			+n^{-\lambda\theta}\sum_{\nu=1}^n \wpar{\frac1\nu}^\theta
				\nu^{(r+\lambda)\theta-1}
		\le\Cn\varphi\left(\frac 1n\right)^\theta,
	\end{displaymath}
	where constant~$\lastC$ does not depend on~$n$.
	Lemma~\ref{lm:anu-w} yields that the last estimate
	is equivalent to the estimate~\cite[p.~31]{bari:trigonometricheskie}
	\begin{multline*}
		\sum_{\nu=n+1}^\infty \nu^{(r-k)\theta-1}
				\biggl(
					\sum_{\mu=1}^\nu a_\mu^p \mu^{(k+1)p-2}
				\biggr)^{\theta/p}
			+\sum_{\nu=n+1}^\infty \nu^{r\theta-1}
				\biggl(
				  \sum_{\mu=\nu}^\infty a_\mu^p \mu^{p-2}
				\biggr)^{\theta/p}\\
		+n^{-\lambda\theta}\sum_{\nu=1}^n \nu^{(r+\lambda-k)\theta-1}
				\biggl(
				  \sum_{\mu=1}^\nu a_\mu^p \mu^{(k+1)p-2}
				\biggr)^{\theta/p}\\
		+n^{-\lambda\theta}\sum_{\nu=1}^n \nu^{(r+\lambda)\theta-1}
				\biggl(
					\sum_{\mu=\nu}^\infty a_\mu^p \mu^{p-2}
				\biggr)^{\theta/p}
			\le\Cn\varphi\left(\frac 1n\right)^\theta,
	\end{multline*}
	where constant~$\lastC$ does not depend on~$n$.
	Hence,
	if we denote the terms on the left--hand side of the inequality
	by~$J_1$, $J_2$, $J_3$ and~$J_4$ respectively,
	then condition $f\in\Np$ is equivalent to the condition
	\begin{equation}\label{eq:J-phi}
		J_1+J_2+J_3+J_4\le\lastC\varphi\left(\frac 1n\right)^\theta.
	\end{equation}
	
	Now we estimate the terms~$J_1$, $J_2$, $J_3$ and~$J_4$
	from below and above
	by means of expression taking part in the condition of the theorem.
	
	First we estimate~$J_1$ and~$J_2$ from below.
	We have
	\begin{multline*}
		J_1=\sum_{\nu=n+1}^\infty \nu^{(r-k)\theta-1}
				\biggl(
				  \sum_{\mu=1}^\nu a_\mu^p \mu^{(k+1)p-2}
				\biggr)^{\theta/p}\\
		\ge\sum_{\nu=n+1}^\infty \nu^{-(k-r)\theta-1}
				\biggl(
				  \sum_{\mu=n+1}^\nu a_\mu^p \mu^{(k+1)p-2}
				\biggr)^{\theta/p}.
	\end{multline*}
	For $k-r>0$,
	making use of Lemmas~\ref{lm:lp} and~\ref{lm:lp-converse}
	we obtain
	\begin{multline}\label{eq:J1ge}
		J_1\ge\Cn\sum_{\nu=4(n+1)}^\infty \nu^{-(k-r)\theta-1}
			(a_\nu^p \nu^{(k+1)p-2 \nu})^{\theta/p}\\
		=\lastC\sum_{\nu=4(n+1)}^\infty a_\nu^\theta
			\nu^{r\theta+\theta-\theta/p-1}.
	\end{multline}
	In an analogous way,
	for $r\theta>0$ we get
	\begin{equation}\label{eq:J2ge}
		J_2=\sum_{\nu=n+1}^\infty \nu^{r\theta-1}
			\biggl(
				\sum_{\mu=\nu}^\infty a_\mu^p \mu^{p-2}
			\biggr)^{\theta/p}
		\ge\Cn\sum_{\nu=8(n+1)}^\infty a_\nu^\theta
		  \nu^{r\theta+\theta-\theta/p-1}.
	\end{equation}
	
	We estimate the term~$J_2$ from above:
	\begin{equation}\label{eq:J2le}
		J_2\le\Cn\sum_{\nu=\left[\frac{n+1}4\right]}^\infty \nu^{r\theta-1}
			(a_\nu^p \nu^{p-2 \nu})^{\theta/p}
		=\lastC\sum_{\nu=\left[\frac{n+1}4\right]}^\infty a_\nu^\theta
			\nu^{r\theta+\theta-\theta/p-1}.
	\end{equation}
	
	For~$J_1$ we have
	\begin{multline*}
		J_1\le\Cn\biggl(\sum_{\nu=n+1}^\infty \nu^{-(k-r)\theta-1}
				\biggl(
				  \sum_{\mu=n+1}^\nu a_\mu^p \mu^{(k+1)p-2}
				\biggr)^{\theta/p}\\
		+\sum_{\nu=n+1}^\infty \nu^{-(k-r)\theta-1}
				\biggl(
				  \sum_{\mu=1}^n a_\mu^p \mu^{(k+1)p-2}
				\biggr)^{\theta/p}\biggr),
	\end{multline*}
	and applying once more Lemmas~\ref{lm:lp} and~\ref{lm:lp-converse}
	we obtain
	\begin{equation}\label{eq:J1}
		J_1\le\Cn\sum_{\nu=\left[\frac{n+1}4\right]}^\infty a_\nu^\theta
				\nu^{r\theta+\theta-\theta/p-1}
			+n^{-(k-r)\theta}
				\biggl(
					\sum_{\mu=1}^n a_\mu^p \mu^{(k+1)p-2}
				\biggr)^{\theta/p}.
	\end{equation}
	Put
	\begin{displaymath}
		I_1=n^{-(k-r)\theta}\sum_{\mu=1}^n a_\mu^p \mu^{(k+1)p-2}.
	\end{displaymath}
	Then for
	\begin{displaymath}
		I_2=I_1n^{(k-r)\theta},
	\end{displaymath}
	taking into account that $(k+1)p-2\ge0$ and $a_\nu\downarrow0$
	we get
	\begin{multline*}
		I_2=\sum_{\mu=1}^n a_\mu^p \mu^{(k+1)p-2}
		\le\sum_{\mu=1}^{\left[\frac n2\right]}
				a_\mu^p \mu^{(k+1)p-2}
			+a_{\left[\frac n2\right]+1}^p
				\sum_{\mu=\left[\frac n2\right]+1}^n \mu^{(k+1)p-2}\\
		\le\sum_{\mu=1}^{\left[\frac n2\right]}
				a_\mu^p \mu^{(k+1)p-2}
			+\Cn n^{(k+1)p-1}a_{\left[\frac n2\right]+1}^p
		\le\Cn\sum_{\mu=1}^{\left[\frac n2\right]} a_\mu^p \mu^{(k+1)p-2}.
	\end{multline*}
	Since $k-r-\lambda>0$,
	we have
	\begin{multline*}
		I_1^{\theta/p}\le\Cn n^{-(k-r)\theta}
			\biggl(
				\sum_{\mu=1}^{\left[\frac n2\right]} a_\mu^p \mu^{(k+1)p-2}
			\biggr)^{\theta/p}\\
		\le\Cn n^{-\lambda\theta}\sum_{\nu=\left[\frac n2\right]}^n
			\nu^{-(k-r-\lambda)\theta-1}
			\biggl(
				\sum_{\mu=1}^\nu a_\mu^p \mu^{(k+1)p-2}
			\biggr)^{\theta/p}\\
		\le\lastC n^{-\lambda\theta}\sum_{\nu=1}^n
			\nu^{-(k-r-\lambda)\theta-1}
			\biggl(
				\sum_{\mu=1}^\nu a_\mu^p \mu^{(k+1)p-2}
			\biggr)^{\theta/p}.
	\end{multline*}
	Applying Lemma~\ref{lm:lp-complete} we obtain
	\begin{multline*}
		I_1^{\theta/p}\le\Cn n^{-\lambda\theta}
			\sum_{\nu=1}^n	\nu^{-(k-r-\lambda)\theta-1}
			(a_\nu^p \nu^{(k+1)p-2} \nu)^{\theta/p}\\
		=\lastC n^{-\lambda\theta} \sum_{\nu=1}^n a_\nu^\theta
			\nu^{(r+\lambda)\theta+\theta-\theta/p-1}.
	\end{multline*}
	From~\eqref{eq:J1} it follows that
	\begin{equation}\label{eq:J1le}
		J_1\le\Cn\biggl(
			\sum_{\nu=\left[\frac{n+1}4\right]}^\infty a_\nu^\theta
					\nu^{r\theta+\theta-\theta/p-1}
				+n^{-\lambda\theta}\sum_{\nu=1}^n a_\nu^\theta
					\nu^{(r+\lambda)\theta+\theta-\theta/p-1}
		\biggr).
	\end{equation}
	
	This way,
	inequalities~\eqref{eq:J1ge}, \eqref{eq:J2ge},
	\eqref{eq:J2le} and~\eqref{eq:J1le}
	yield
	\begin{multline}\label{eq:J1-J2}
		\Cn\sum_{\nu=8(n+1)}^\infty
				a_\nu^\theta
				\nu^{r\theta+\theta-\theta/p-1}
			\le J_1+J_2\\
		\le\Cn\biggl(
				\sum_{\nu=\left[\frac{n+1}4\right]}^\infty
					a_\nu^\theta
					\nu^{r\theta+\theta-\theta/p-1}
				+n^{-\lambda\theta}\sum_{\nu=1}^n a_\nu^\theta
					\nu^{(r+\lambda)\theta+\theta-\theta/p-1}
			\biggr).
	\end{multline}
	
	Now we estimate~$J_3$ and~$J_4$. Put
	\begin{displaymath}
		A_1=n^{\lambda\theta}J_3
		=\sum_{\nu=1}^n \nu^{(r+\lambda-k)\theta-1}
			\biggl(
				\sum_{\mu=1}^\nu a_\mu^p \mu^{(k+1)p-2}
			\biggr)^{\theta/p}
	\end{displaymath}
	and
	\begin{displaymath}
		A_2
		=n^{\lambda\theta}J_4=\sum_{\nu=1}^n \nu^{(r+\lambda)\theta-1}
			\biggl(
				\sum_{\mu=\nu}^\infty a_\mu^p \mu^{p-2}
			\biggr)^{\theta/p},
	\end{displaymath}
	applying Lemma~\ref{lm:lp-complete} for $r+\lambda-k<0$
	we get
	\begin{equation}\label{eq:A1}
		A_1
		\le\Cn\sum_{\nu=1}^n a_\nu^\theta
			\nu^{(r+\lambda)\theta+\theta-\theta/p-1}.
	\end{equation}
	
	We estimate~$A_2$ in an analogous way:
	\begin{multline}\label{eq:A2}
		A_2\le\Cn\biggl(
				\sum_{\nu=1}^n \nu^{(r+\lambda)\theta-1}
					\biggl(
					  \sum_{\mu=\nu}^n a_\mu^p \mu^{p-2}
					\biggr)^{\theta/p}\\
		+\sum_{\nu=1}^n \nu^{(r+\lambda)\theta-1}
				\biggl(
				  \sum_{\mu=n+1}^\infty a_\mu^p \mu^{p-2}
				\biggr)^{\theta/p}
			\biggr)\\
		\le\Cn\biggl(
				\sum_{\nu=1}^n a_\nu^\theta
					\nu^{(r+\lambda)\theta+\theta-\theta/p-1}
				+n^{(r+\lambda)\theta}
					\biggl(
					  \sum_{\mu=n+1}^\infty a_\mu^p \mu^{p-2}
					\biggr)^{\theta/p}
			\biggr).
	\end{multline}
	
	We estimate the series
	\begin{displaymath}
		B=\biggl(\sum_{\mu=n+1}^\infty a_\mu^p \mu^{p-2}\biggr)^{\theta/p}.
	\end{displaymath}
	
	First let $\frac\theta p>1$.
	Applying H\"older inequality we have
	\begin{multline*}
		\sum_{\mu=n+1}^\infty
			a_\mu^p \mu^{p-2}
		\le\biggl(
			\sum_{\mu=n+1}^\infty(a_\mu^p \mu^{p-1+rp-p/\theta})^{\theta/p}
		\biggr)^{p/\theta}\\
		\times\biggl(
			\sum_{\mu=n+1}^\infty\bigl(
				\mu^{-(rp-p/\theta+1)\theta/(\theta-p)}
		\biggr)^{(\theta-p)/\theta}.
	\end{multline*}
	Since
	$\bigl(rp-\frac p\theta+1\bigr)\frac\theta{\theta-p}
	=rp\frac\theta{\theta-p}+1>1$,
	we get
	\begin{displaymath}
		\sum_{\mu=n+1}^\infty a_\mu^p \mu^{p-2}
		\le\Cn n^{-rp}\biggl(
			\sum_{\mu=n+1}^\infty a_\mu^\theta
					\mu^{\theta-\theta/p+r\theta-1}
		\biggr)^{p/\theta}.
	\end{displaymath}
	So, for $\frac\theta p>1$ we have proved that
	\begin{displaymath}
		B\le\Cn n^{-r\theta}\sum_{\mu=n+1}^\infty a_\mu^\theta
			\mu^{r\theta+\theta-\theta/p-1}.
	\end{displaymath}
	
	Let $\frac\theta p\le1$.
	For given~$n$ we choose the positive integer~$N$
	such that $2^N\le n+1<2^{N+1}$.
	Then we have
	\begin{multline*}
		B\le\biggl(\sum_{\mu=2^N}^\infty a_\mu^p \mu^{p-2}\biggr)^{\theta/p}
			\le\biggl(
				\sum_{\nu=N}^\infty a_{2^\nu}^p
				\sum_{\mu=2^\nu}^{2^{\nu+1}-1}\mu^{p-2}
			\biggr)^{\theta/p}\\
		\le\Cn\biggl(
			\sum_{\nu=N}^\infty a_{2^\nu}^p 2^{\nu(p-1)}
		\biggr)^{\theta/p}.
	\end{multline*}
	Making use of Lemma~\ref{lm:jensen}
	we obtain
	\begin{multline*}
		B\le\lastC\sum_{\nu=N}^\infty a_{2^\nu}^\theta
			2^{\nu(\theta-\theta/p)}
		\le\Cn\sum_{\nu=N}^\infty \sum_{\mu=2^{\nu-1}}^{2^\nu-1}
			a_\mu^\theta \mu^{\theta-\theta/p-1}\\
		=\lastC\sum_{\nu=2^{N-1}}^\infty a_\nu^\theta
			\nu^{\theta-\theta/p-1}
		\le\lastC\sum_{\nu=\left[\frac{n+1}4\right]}^\infty
			a_\nu^\theta \nu^{\theta-\theta/p-1}\\
		\le\lastC\left[\frac{n+1}4\right]^{-r\theta}
			\sum_{\nu=\left[\frac{n+1}4\right]}^\infty a_\nu^\theta
			\nu^{r\theta+\theta-\theta/p-1}.
	\end{multline*}
	Since for $n\ge3$ holds $\left[\frac{n+1}4\right]\ge\frac n{12}$,
	we get
	\begin{displaymath}
		B\le\Cn n^{-r\theta}
			\sum_{\nu=\left[\frac{n+1}4\right]}^\infty a_\nu^\theta
			\nu^{r\theta+\theta-\theta/p-1}.
	\end{displaymath}
	
	This way, for $0<\frac\theta p<\infty$ we proved that
	\begin{displaymath}
		B\le\Cn n^{-r\theta}
			\sum_{\nu=\left[\frac{n+1}4\right]}^\infty a_\nu^\theta
			\nu^{r\theta+\theta-\theta/p-1}.
	\end{displaymath}
	Hence~\eqref{eq:A2} yields
	\begin{displaymath}
		A_2
		\le\Cn\biggl(
			\sum_{\nu=1}^n a_\nu^\theta
				\nu^{(r+\lambda)\theta+\theta-\theta/p-1}
			+n^{\lambda\theta}\sum_{\nu=\left[\frac{n+1}4\right]}^\infty
				a_\nu^\theta \nu^{r\theta+\theta-\theta/p-1}
		\biggr).
	\end{displaymath}
	
	Now, from~\eqref{eq:A1} it follows that
	\begin{multline}\label{eq:J3-J4}
		J_3+J_4=n^{-\lambda\theta}(A_1+A_2)\\
		\le\Cn\biggl(
				n^{-\lambda\theta}\sum_{\nu=1}^n a_\nu^\theta
					\nu^{(r+\lambda)\theta+\theta-\theta/p-1}
				+\sum_{\nu=\left[\frac{n+1}4\right]}^\infty
					a_\nu^\theta \nu^{r\theta+\theta-\theta/p-1}
		\biggr).
	\end{multline}
	
	Further, we estimate the series
	\begin{displaymath}
		A_3=\sum_{\nu=\left[\frac{n+1}4\right]}^\infty
			a_\nu^\theta \nu^{r\theta+\theta-\theta/p-1}
		=A_4+\sum_{\nu=n+1}^\infty	a_\nu^\theta
			\nu^{r\theta+\theta-\theta/p-1},
	\end{displaymath}
	where is
	\begin{multline*}
		A_4=\sum_{\nu=\left[\frac{n+1}4\right]}^n
				a_\nu^\theta \nu^{r\theta+\theta-\theta/p-1}
			\le\Cn a_{\left[\frac{n+1}4\right]}^\theta
				n^{r\theta+\theta-\theta/p}\\
		\le\Cn n^{-\lambda\theta}
				\sum_{\nu=1}^{\left[\frac{n+1}4\right]} a_\nu^\theta
					\nu^{(r+\lambda)\theta+\theta-\theta/p-1}
		\le\lastC n^{-\lambda\theta}
				\sum_{\nu=1}^n a_\nu^\theta
					\nu^{(r+\lambda)\theta+\theta-\theta/p-1}.
	\end{multline*}
	Whence
	\begin{equation}\label{eq:A3}
		A_3\le\Cn\biggl(
			n^{-\lambda\theta}\sum_{\nu=1}^n a_\nu^\theta
				\nu^{(r+\lambda)\theta+\theta-\theta/p-1}
			+\sum_{\nu=n+1}^\infty	a_\nu^\theta
				\nu^{r\theta+\theta-\theta/p-1}
		\biggr).
	\end{equation}
	
	Making use of~\eqref{eq:A3} and~\eqref{eq:J3-J4}
	we have
	\begin{displaymath}
		J_3+J_4\le\Cn\biggl(
			n^{-\lambda\theta}\sum_{\nu=1}^n a_\nu^\theta
				\nu^{(r+\lambda)\theta+\theta-\theta/p-1}
			+\sum_{\nu=n+1}^\infty	a_\nu^\theta
				\nu^{r\theta+\theta-\theta/p-1}
		\biggr).
	\end{displaymath}
	Hence, applying~\eqref{eq:A3} in~\eqref{eq:J1-J2}
	we obtain
	\begin{multline}\label{eq:J1-J2-J3-J4le}
		J_1+J_2+J_3+J_4\\
		\le\Cn\biggl(
			n^{-\lambda\theta}\sum_{\nu=1}^n a_\nu^\theta
				\nu^{(r+\lambda)\theta+\theta-\theta/p-1}
			+\sum_{\nu=n+1}^\infty	a_\nu^\theta
				\nu^{r\theta+\theta-\theta/p-1}
		\biggr).
	\end{multline}
	
	Now we estimate~$A_1$ and~$A_2$ from below.
	Making use of Lemma~\ref{lm:lp-complete}
	we get
	\begin{displaymath}
		A_1\ge\Cn\sum_{\nu=1}^n a_\nu^\theta
			\nu^{(r+\lambda)\theta+\theta-\theta/p-1},
	\end{displaymath}
	and in an analogous way
	\begin{displaymath}
		A_2\ge\sum_{\nu=1}^n \nu^{(r+\lambda)\theta-1}
			\biggl(
				\sum_{\mu=\nu}^n a_\mu^p \mu^{p-2}
			\biggr)^{\theta/p}
		\ge\Cn\sum_{\nu=1}^n a_\nu^\theta
			\nu^{(r+\lambda)\theta+\theta-\theta/p-1}.
	\end{displaymath}
	Hence
	\begin{displaymath}
		A_1+A_2\ge\Cn\sum_{\nu=1}^n a_\nu^\theta
			\nu^{(r+\lambda)\theta+\theta-\theta/p-1}.
	\end{displaymath}
	This way the following inequality holds
	\begin{displaymath}
		J_3+J_4\ge\Cn n^{-\lambda\theta}\sum_{\nu=1}^n a_\nu^\theta
			\nu^{(r+\lambda)\theta+\theta-\theta/p-1}.
	\end{displaymath}
	
	From~\eqref{eq:J1-J2} it follows that
	\begin{multline}\label{eq:J1-J2-J3-J4ge}
		J_1+J_2+J_3+J_4\\
		\ge\Cn\biggl(
			\sum_{\nu=8(n+1)}^\infty a_\nu^\theta
				\nu^{r\theta+\theta-\theta/p-1}
			+n^{-\lambda\theta}\sum_{\nu=1}^n a_\nu^\theta
				\nu^{(r+\lambda)\theta+\theta-\theta/p-1}
		\biggr).
	\end{multline}
	Since
	\begin{multline*}
		\sum_{\nu=n+1}^{\nu=8(n+1)-1} a_\nu^\theta
				\nu^{r\theta+\theta-\theta/p-1}
			\le\Cn a_n^\theta	n^{r\theta+\theta-\theta/p}\\
		\le\Cn n^{-\lambda\theta}\sum_{\nu=1}^n a_\nu^\theta
				\nu^{(r+\lambda)\theta+\theta-\theta/p-1}
	\end{multline*}
	holds,
	we have
	\begin{multline*}
		\sum_{\nu=n+1}^\infty a_\nu^\theta
				\nu^{r\theta+\theta-\theta/p-1}
			+n^{-\lambda\theta}\sum_{\nu=1}^n a_\nu^\theta
				\nu^{(r+\lambda)\theta+\theta-\theta/p-1}\\
		\le\Cn\biggl(
			\sum_{\nu=8(n+1)}^\infty a_\nu^\theta
				\nu^{r\theta+\theta-\theta/p-1}
			+n^{-\lambda\theta}\sum_{\nu=1}^n a_\nu^\theta
				\nu^{(r+\lambda)\theta+\theta-\theta/p-1}
		\biggr).
	\end{multline*}
	
	Now,
	estimates~\eqref{eq:J1-J2-J3-J4ge} and~\eqref{eq:J1-J2-J3-J4le}
	imply
	\begin{multline*}
		\Cn\biggl(
			\sum_{\nu=n+1}^\infty a_\nu^\theta
				\nu^{r\theta+\theta-\theta/p-1}
			+n^{-\lambda\theta}\sum_{\nu=1}^n a_\nu^\theta
				\nu^{(r+\lambda)\theta+\theta-\theta/p-1}
		\biggr)\\
		\le J_1+J_2+J_3+J_4\\
		\le\Cn\biggl(
			\sum_{\nu=n+1}^\infty a_\nu^\theta
				\nu^{r\theta+\theta-\theta/p-1}
			+n^{-\lambda\theta}\sum_{\nu=1}^n a_\nu^\theta
				\nu^{(r+\lambda)\theta+\theta-\theta/p-1}
		\biggr).
	\end{multline*}
	
	This way we proved that condition~\eqref{eq:Np-w}
	is equivalent to the condition of the theorem.
	Since condition~\eqref{eq:Np-w}
	is equivalent to the condition $f\in\Np$,
	proof of Theorem~\ref{th:Np-monotone} is completed.
\end{proof}

\bibliographystyle{amsplain}
\bibliography{maths,toappear}

\end{document}